\documentclass{amsart}
\usepackage[utf8]{inputenc}
\usepackage{amssymb}
\usepackage{amsthm}
\usepackage[colorlinks]{hyperref}
\usepackage{enumitem}

\numberwithin{equation}{section}

\newtheorem{thm}{Theorem}
\newtheorem{lemma}[thm]{Lemma}
\newtheorem{cor}[thm]{Corollary}
\newtheorem{rmk}[thm]{Remark}
\newtheorem{prop}[thm]{Proposition}
\theoremstyle{definition}

\newtheorem{example}[thm]{Example}

\begin{document}
\title[nowhere differentiable functions of the Generalized Takagi Class]{A characterization  of the nowhere differentiable functions of the Generalized Takagi Class}

\date{\today\ (\the\time)}

\author{Juan Ferrera}
\address{IMI, Departamento de An{\'a}lisis Matem{\'a}tico y Matem{\'a}tica Aplicada,
Facultad Ciencias Matem{\'a}ticas, Universidad Complutense, 28040, Madrid, Spain}
\email{ferrera@mat.ucm.es}

\author{Javier G\'omez Gil}
\address{Departamento de An{\'a}lisis Matem{\'a}tico y Matem{\'a}tica Aplicada,
Facultad Ciencias Matem{\'a}ticas, Universidad Complutense, 28040, Madrid, Spain}
\email{gomezgil@mat.ucm.es}

\author{Jes\'us Llorente}
\address{Departamento de An{\'a}lisis Matem{\'a}tico y Matem{\'a}tica Aplicada,
Facultad Ciencias Matem{\'a}ticas, Universidad Complutense, 28040, Madrid, Spain}
\email{jesllore@ucm.es}


\begin{abstract}
In this paper, we prove that for some Generalized Takagi Classes, in particular for the
Takagi-Van der Waerden Class, the functions are nowhere differentiable if, and only if, the sequence of weights 
does not belong to $c_0$.
\end{abstract}

\maketitle

\section{Introduction}

The Takagi function introduced in \cite{T} provides an example of a continuous nowhere differentiable function
(see \cite{AK} or \cite{L} for a wide introduction to this funciton).
It is usually defined by 
$$
T(x)=\sum_{n=0}^{\infty}\frac{\phi(2^nx)}{2^n}, \quad x\in [0,1],
$$
where $\phi(x)$ denotes the distance from the point $x$ to the nearest integer. Although it is simpler than the celebrated Weierstrass
function, it is more regular than that one, in the sense that it may have infinite lateral derivatives (see 
\cite{AK1} for a study of this matter).

From the subdifferential viewpoint, the Takagi function is also interesting  since it is extremal in the following
sense: it has empty subdifferential except for a countable set, where its subdifferential is $\mathbb{R}$
(see \cite{G}).

Recently, the first two authors, see \cite{FGG}, have generalized the Takagi function to separable Hilbert spaces, considering arbitrary  
countable dense subsets $D$. More precisely, for the one dimensional case, they consider an increasing sequence $\mathcal{D}=(D_n)_n$ of finite subsets of $[0,1]$ satisfying that $D=\cup_n D_n$ and $dist(z,D_n)\leq \alpha _n$ for every $z\in [0,1]$, where $\alpha =(\alpha _n)_n\in \ell ^1$. They define the Takagi function $T_{\mathcal{D}}:[0,1]\to\mathbb{R}$ associated to $\mathcal{D}$ as
$$
T_{\mathcal{D}}(x)=\sum_{n=0}^{\infty}g_n(x)
$$
where $g_n(x)$ denotes the distance from $x$ to the set $D_n$. They prove that these functions share with the
Takagi function the following property: the subdifferential of $T_{\mathcal{D}}$ is $\mathbb{R}$ at every point $x\in D$, meanwhile it is empty provided that $x\notin D$ (see \cite{FGG}).

From the Takagi function, Hata and Yamaguti \cite{HY} introduced the Takagi Class consisting in the functions
of the form
$$
f_w(x)=\sum_{n=0}^{\infty}w_n\frac{\phi(2^nx)}{2^n},
$$
where the sequence $w=(w_n)_n$ satisfies that the series $\sum_n\frac{w_n}{2^n}$ converges.
They proved, among other results, that the Takagi Class is a closed subspace of the space of
continuous functions $C[0,1]$ endowed with the sup norm.

Kôno \cite{Kono} studied the differentiability of the functions of the Takagi Class, proving in particular that
$f_w$ is nowhere derivable if and only if $w\notin c_0$.

We may generalize the Takagi Class in the following way: we consider a countable dense subset $D$ of $[0,1]$ and we assume that $0,1\in D$. Let $(D_n)_n$ be a decomposition of $D$, that is an increasing sequence of finite sets $D_n$ and a 
sequence $(\alpha _n)_n\in \ell ^1$ satisfying: 
\begin{enumerate}
  \item $0,1\in D_0$. 
  \item $D=\cup_{n=0}^{\infty}D_n$.
  \item $|x-y|\leq \alpha _n$ for every $x,y\in D_n$ such that $(x,y)\cap D_n =\emptyset$.
  \item There exists $\rho \in (0,1]$ such that $|x-y|\geq \rho \alpha _n$ for every $x,y\in D_n$,
  $x\neq y$.
\end{enumerate}
The Generalized Takagi Class 
is composed by the functions $T_w: [0,1]\to \mathbb{R}$ defined by
$$
T_w(x)=\sum_{n=0}^{\infty}w_ng_n(x)
$$
where $w=(w_n)_n$ satisfies that $(w_n\alpha _n)_n\in \ell ^1$.

This generalization was introduced in \cite{FGG1}, where the authors prove that, under mild conditions on the 
decomposition, the Generalized Takagi Class is a subspace of $C[0,1]$ isomorphic to $\ell ^1$. They also
study the differentiability of the functions of the Class when the sequence of weights $w$  belongs
to $c_0$, generalizing the Kôno's results in particular.

We introduce some notation. We denote by $\mathcal{F}_n$ the family of connected components of 
$[0,1]\smallsetminus D_n$, and by $\widetilde{D}_n$ the set of midpoints of the intervals $I\in \mathcal{F}_n$. Finally, $\widetilde{D}$ will denote the union of the sets $\widetilde{D}_n$.

The aim of this paper is to prove that under some restrictions on the decomposition, the functions of the class 
are nowhere derivable provided that $w\notin c_0$. More precisely, we will prove the following Main Theorem.
\begin{thm}\label{main}
If the decomposition satisfies  for every $n$ that either
\begin{enumerate}
  \item  $\widetilde{D}_n\subset D_{n+1}$, or 
    \item  $I\cap D_{n+1}\neq \emptyset$ for every $I\in \mathcal{F}_n$ and $\widetilde{D}_n\subset \widetilde{D}_{n+1}$,
\end{enumerate}
then $T_w$ is nowhere differentiable if and only if $w\notin c_0$.
\end{thm}

As in \cite{FGG1} it is proved that, under more general restrictions on the decomposition, $T_w$ is derivable at an
uncountable set of points provided that $w\in c_0$, this Main Theorem will follow from Theorem \ref{rpar} below.

The next example of a Generalized  Takagi Class, named Takagi-Van der Waerden Class, shows how wide is the range
of the result that we present in this paper.

\begin{example}\label{ejemplo1}
Let $r\geq 2$ be an integer. The Takagi-Van der Waerden function $f_r:[0,1]\to \mathbb{R}$ is defined as follows
$$
f_r(x)=\sum_{n=0}^{\infty} \frac{1}{r^n}\phi(r^n x).
$$
Let us observe that $f_2$ is the Takagi function (see \cite{T}) and $f_{10}$ is the Van der Waerden function (see \cite{Waerden}). The Takagi-Van der Waerden function is nowhere differentiable (see \cite{A} or \cite{BE}, for instance) and the set of points at which $f_r$ has an infinite derivative is characterized in \cite{FGGLL}.

In this case, it is clear that $D_n = \{ kr^{-n}\in [0,1]:k\in\mathbb{Z}\}$, $\alpha_n = \frac{1}{r^{n}}$ and $\rho=1$. When $r$ is even we have the situation $\widetilde{D}_n\subset D_{n+1}$, meanwhile the situacion $\widetilde{D}_n\subset \widetilde{D}_{n+1}$ arises when $r$ is odd. 

The Takagi-Van der Waerden Class is an immediate generalization of the Takagi Class. It is formed by  all the functions $f_{r,w}:[0,1]\to\mathbb{R}$ defined by
$$
f_{r,w}(x)=\sum_{n=0}^{\infty}\frac{w_n}{r^n}\phi(r^nx)
$$
where $w$ satisfies that $\sum_{n=0}^{\infty}\frac{w_n}{r^n}<+\infty$. As a consequence of Theorem \ref{main},
we have that $f_{r,w}$ is nowhere differentiable 
whenever $w\not\in c_0$.
\end{example}

We proceed to introduce some definitions and results on nonsmooth analysis. For a lower semicontinuous
function $f:\mathbb{R}\to \mathbb{R}$ we define the subdifferential of $f$ at $x$, $\partial f(x)$, as the set
of real numbers $c$ that satisfy 
$$
\liminf_{h\to 0}\frac{f(x+h)-f(x)-ch}{|h|}\geq 0.
$$ 
If we consider the Dini derivatives,
$$
D^-f(x)=\limsup_{t\to 0^-}\frac{f(x+t)-f(x)}{t} \quad\text{and}\quad
d_+f(x)=\liminf_{t\to 0^+}\frac{f(x+t)-f(x)}{t},
$$
then
$$
\partial f(x)=[D^-f(x),d_+f(x)]\cap \mathbb{R}.
$$
The superdifferential $\partial ^+f(x)$ of an upper semicontinuous function $f$ at $x$,
may be defined by the following formula:
$$
\partial ^+f(x)=-\partial (-f)(x).
$$
A function $f$ is derivable at $x$ if and only if $\partial f(x)=\partial ^+f(x)\neq \emptyset$. 
If this is the case, then $\partial f(x)=\{ f'(x)\}$.
For a quick introduction to these concepts see \cite{F}.

We end this introduction with a lemma that is probably
well known, however we prove it for the sake of selfcontainess.

\begin{lemma}\label{lema1}
Let $I\subset \mathbb{R}$ be an open interval, 
$f:I\to \mathbb{R}$ a function, and $x\in I$.
\begin{enumerate}
  \item Let $a_n<x<b_n$ for every $n$. Assume that $\lim_na_n=\lim_nb_n=x$.
If $f$ is derivable at $x$, then
$$
\lim_n \frac{f(b_n)-f(a_n)}{b_n-a_n}=f'(x).
$$
  \item  Let $x<u_n<v_n$ for every $n$.
  Assume that  $\lim_nv_n=x$, and $\limsup_n\frac{u_n-x}{v_n-u_n}<+\infty$.
  If $f$ is right derivable at $x$, then
$$
\lim_n \frac{f(v_n)-f(u_n)}{v_n-u_n}=f'^+(x).
$$
  \item Let  $u_n<v_n<x$ for every $n$. Assume that
  $\lim_nu_n=x$, and $\limsup_n\frac{x-u_n}{v_n-u_n}<+\infty$.
  If $f$ is left derivable at $x$, then
$$
\lim_n \frac{f(v_n)-f(u_n)}{v_n-u_n}=f'^-(x).
$$
\end{enumerate}
\end{lemma}
\begin{proof}
Let us prove $(1)$.
For every $n$ we denote
$$
r_n=\frac{f(b_n)-f(x)}{b_n-x} \quad \textrm{and} \quad
s_n=\frac{f(a_n)-f(x)}{a_n-x}.
$$
We have $\lim_nr_n=\lim_ns_n=f'(x)$. Since
\begin{align*}
&f(b_n)-f(a_n)=f(b_n)-f(x)+f(x)-f(a_n)\\
&=r_n(b_n-x)-s_n(a_n-x)=r_n(b_n-a_n)+(s_n-r_n)(x-a_n)
\end{align*}
we obtain
$$
\frac{f(b_n)-f(a_n)}{b_n-a_n}=r_n+(s_n-r_n)\frac{x-a_n}{b_n-a_n}.
$$
Taking limits we get the result. 

Let us prove $(2)$.
For every $n$ we denote
$$
r_n=\frac{f(v_n)-f(x)}{v_n-x} \quad \textrm{and} \quad
s_n=\frac{f(u_n)-f(x)}{u_n-x}.
$$
We have $\lim_nr_n=\lim_ns_n=f'^+(x)$. Since
\begin{align*}
&f(v_n)-f(u_n)=f(v_n)-f(x)+f(x)-f(u_n)\\
&=r_n(v_n-x)-s_n(u_n-x)=r_n(v_n-u_n)+(s_n-r_n)(x-u_n)
\end{align*}
we obtain
$$
\frac{f(v_n)-f(u_n)}{v_n-u_n}=r_n+(r_n-s_n)\frac{u_n-x}{v_n-u_n}.
$$
Taking limits we get the result. 

The proof of $(3)$ is similar. 
\end{proof}

\section{Non derivability on $D$}

If $x\in D$, then there exists $n_0\geq 1$ (we are assuming that $x\in (0,1)$ since the study 
of $T'^+_w(0)$ and $T'^-_w(1)$ is similar)
such that $x\in D_{n_0}$
but $x\notin D_{n_0-1}$. 
The function 
$$
G_{n_0-1}=\sum_{k=0}^{n_0-1}w_kg_k
$$ 
has lateral derivatives  at $x$: 
$$
G'^+_{n_0-1}(x)=\sum_{k=0}^{n_0-1}w_kg'^+_k(x) \quad \text{and} \quad G'^-_{n_0-1}(x)=\sum_{k=0}^{n_0-1}w_kg'^-_k(x).
$$
Consequently, in order 
to study the existence of lateral derivatives of $T_w$ we may assume that $x\in D_1$ and $w_0=0$.

Let $x\in D_1$, for every $n\geq 1$ we define
$y_n\in D_n$ satisfying that  $(x,y_n)\in \mathcal{F}_n$.
We will assume that the sequence $(y_n)_n$ is strictly decreasing. 
Observe that this holds when $I\cap D_{n+1}\neq \emptyset$ for every $I\in \mathcal{F}_n$.
This condition is not too restrictive, it means that
between two points of $D_n$ there is always a point of $D_{n+1}$. 
In particular, this condition holds in the interesting case
of Takagi Van der Waerden Class.

With the aim of extending Kôno's Theorem to the Takagi-Van der Waerden Class, we present an easy first 
result that covers this situation.

\begin{prop}\label{sencillo}
Assume that the decomposition satisfies that
for every $n$, either
\begin{enumerate}
  \item  $\widetilde{D}_n\subset D_{n+1}$, or 
  \item  $I\cap D_{n+1}\neq \emptyset$ for every $I\in \mathcal{F}_n$ and $\widetilde{D}_n\subset \widetilde{D}_{n+1}$.
\end{enumerate}
Then, $T_w$ has not finite lateral derivatives at $x\in D$ provided that $w\notin c_0$.
\end{prop}
\begin{proof}
We may assume without loss of generality that $x\in D_1$ and $w_0=0$. It is immediate to see that
$$
\frac{T_w(y_n)-T_w(x)}{y_n-x}=\sum_{k=1}^{n-1}w_k\frac{g_k(y_n)}{y_n-x}=\sum_{k=1}^{n-1}w_k
$$
since $g_k(y_n)=y_n-x$ for every $k<n$. Indeed, if $g_k(y_n)= y_k -y_n$ for some $k<n$ then the midpoint of $(x,y_k)$ belongs to $(x,y_n)$ and consequently, we are necessarily under the second alternative. This implies that $y_{k+1}\in (x,y_n)$ which is a contradiction.
Hence, if $T'^+_w(x)$ exists then the series $\sum w_k$ converges which implies $w\in c_0$.
The proof for the left derivative is similar.
\end{proof}

We set the following lemmas in order to obtain a different kind of conditions.

\begin{lemma}\label{lema2}
Assume that the decomposition satisfies
$\alpha _{n+1}\leq \frac{\rho}{1-\rho} \alpha _n$
for every $n$.
If $x\in D_1$,  $w_0=0$, and $y_{n+1}<y_n$ for every $n$, then  
$$
\Delta _n:=\frac{T_w(y_n)-T_w(x)}{y_n-x}=\sum_{k=1}^{n-2}w_k+\delta _{n-1}w_{n-1}
$$
with $\delta _n\in [\rho ,1]$ for every $n$.
\end{lemma}  
\begin{proof}
We have that
$$
\frac{T_w(y_n)-T_w(x)}{y_n-x}=\sum_{k=1}^{n-1}w_k\frac{g_k(y_n)}{y_n-x}.
$$
If $g_{n-1}(y_n)=y_n-x$ then $g_k(y_n)=y_n-x$ for every $k<n$ too, hence
$$
\sum_{k=1}^{n-1}w_k\frac{g_k(y_n)}{y_n-x}=\sum_{k=1}^{n-1}w_k
$$
and the claim holds with $\delta _{n-1}=1$. Otherwise $g_{n-1}(y_n)=y_{n-1}-y_n$. 
If $g_{n-2}(y_n)=y_{n-2}-y_n$ then
$$
\alpha _n\geq y_n-x> y_{n-2}-y_n=y_{n-2}-y_{n-1}+y_{n-1}-y_n\geq \rho (\alpha _{n-1}+\alpha _n)
$$
which is impossible. 
We deduce that $g_k(y_n)=y_n-x$ for every $k\leq n-2$ which gives us the result.
\end{proof}

\begin{rmk}
Observe that $\delta _{n-1}=\frac{y_{n-1}-y_n}{y_n-x}\geq \rho$
when $g_{n-1}(y_n)=y_{n-1}-y_n$.  
\end{rmk}

\begin{lemma}\label{lema3}
Assume that 
$\alpha _{n+1}\leq \frac{\rho}{1-\rho} \alpha _n$
for every $n$. Let $x\in D_1$,  $w_0=0$, and $y_{n+1}<y_n$ for every $n$.
If $\delta _{n-1}<1$ and $\delta _n<1$  then  
$$
\Gamma _n:=\frac{T_w(y_{n})-T_w(y_{n+1})}{y_{n}-y_{n+1}}=
\sum_{k=1}^{n-2}w_k+\frac{y_{n-1}-y_n-(y_{n+1}-x)}{y_n-y_{n+1}}w_{n-1}-w_{n}.
$$
\end{lemma}
\begin{proof}
It is enough to observe that $g_n(y_{n+1})=y_n-y_{n+1}$,  $g_{n-1}(y_n)=y_{n-1}-y_n$, 
$g_{n-1}(y_{n+1})=y_{n+1}-x$, and $g_k$ is linear on $[y_{n+1},y_n]$ for every $k\leq n-2$.
\end{proof}

\begin{prop}\label{rogrande}
Let $w\notin c_0$.
Assume that $\rho >\frac{1}{2}$.
If $x\in D_1$, $w_0=0$, and   $y_{n+1}<y_n$ for every $n$, then
$T_w$ is not right derivable at $x$. 
\end{prop}
\begin{proof}
If $T_w$ is right derivable at $x$ then
$$
T_w'^+(x)=\lim_n\Delta _n.
$$
Observe that $\rho >\frac{1}{2}$ implies $\alpha _{n+1}\leq \frac{\rho}{1-\rho} \alpha _n$.
Hence, by Lemma \ref{lema2},
$$
0=\lim_n (\Delta _{n+1}-\Delta _n)=
\lim_n\bigl( (1-\delta _{n-1})w_{n-1}+\delta _nw_n\bigr).
$$
We denote
$$
\eta _n:=\delta _nw_n+(1-\delta _{n-1})w_{n-1}
$$
and we have
$$
|w_n|\leq \frac{1}{\delta _n}|\eta _n|+\frac{1-\delta _{n-1}}{\delta _n}|w_{n-1}|
\leq \frac{1}{\rho}|\eta _n|+\frac{1-\rho}{\rho}|w_{n-1}|.
$$
Hence, $w\in c_0$.
\end{proof}

\begin{prop}\label{xenD}
Let $w\notin c_0$.
Assume that $\alpha _{n+1}\leq \rho \alpha _n$
for every $n$.
If $x\in D_1$, $w_0=0$, and   $y_{n+1}<y_n$ for every $n$, then
$T_w$ is not right derivable at $x$. 
\end{prop}
\begin{proof}
Observe first that  $\alpha _{n+1}\leq \rho \alpha _n$
implies $\alpha _{n+1}\leq \frac{\rho}{1-\rho} \alpha _n$.
If $T_w$ is right derivable at $x$ then
$$
T_w'^+(x)=\lim_n\Delta _n=\lim_n \Gamma _n,
$$
where the second equality follows from Lemma \ref{lema1}. Hence, we have 
$$
0=\lim_n (\Delta _{n+1}-\Delta _n)=
\lim_n\bigl( (1-\delta _{n-1})w_{n-1}+\delta _nw_n\bigr)
$$
by Lemma \ref{lema2}. As before, we denote
$$
\eta _n:=\delta _nw_n+(1-\delta _{n-1})w_{n-1}.
$$
If $\delta _{n-1}=1$ then
$$
|w_n|=\frac{1}{\delta _n}|\eta _n|\leq \frac{1}{\rho}|\eta _n|.
$$
If $\delta _n=1$ then
$$
|w_n|\leq |\eta _n|+(1-\delta _{n-1})|w_{n-1}|\leq |\eta _n|+(1-\rho )|w_{n-1}|.
$$

We restrict to the situation $\delta _{n-1}<1$ and $\delta _n<1$.
If we denote $\lambda _n:=\Delta _n-\Gamma _n$, we have
that $\lim_n\lambda _n=0$, and 
$$
\lambda _n=
w_{n-1}\left( \delta _{n-1}-\frac{y_{n-1}-y_n-(y_{n+1}-x)}{y_n-y_{n+1}}\right) +w_n
$$
by Lemma \ref{lema3}.
It is immediate that
$$
\frac{y_{n-1}-y_n-(y_{n+1}-x)}{y_n-y_{n+1}}<\delta _{n-1}
$$
since
$$
\frac{y_{n-1}-y_n}{y_n-y_{n+1}}-\delta _{n-1}=\delta _{n-1}\left( \frac{y_n-x}{y_n-y_{n+1}}-1\right)
=\delta _{n-1}\left( \frac{y_{n+1}-x}{y_n-y_{n+1}}\right) <\frac{y_{n+1}-x}{y_n-y_{n+1}}.
$$
If $\delta _{n-1}\geq 1-\frac{\rho}{2}$ then
\begin{align*}
&0<\delta _{n-1}-\frac{y_{n-1}-y_n-(y_{n+1}-x)}{y_n-y_{n+1}}=
\delta _{n-1}+\frac{1}{\delta _n}-\delta _{n-1}\frac{y_n-x}{y_n-y_{n+1}}\\
&=\frac{1}{\delta _n}+\delta _{n-1}\frac{x-y_{n+1}}{y_n-y_{n+1}}=
\frac{1}{\delta _n}-\frac{\delta _{n-1}}{\delta _n}\leq \frac{\rho}{2\delta _n}\leq \frac{1}{2},
\end{align*}
and therefore $|w_n|\leq |\lambda _n|+\frac{1}{2}|w_{n-1}|$. 

Alternatively, if $\delta _{n-1}\leq 1-\frac{\rho}{2}$,
as $\alpha _{n+1}\leq \rho \alpha _n$ implies that 
$$
\frac{y_{n-1}-y_n-(y_{n+1}-x)}{y_n-y_{n+1}}\geq 0,
$$
we have that  $|w_n|\leq |\lambda _n|+\delta _{n-1}|w_{n-1}|\leq |\lambda _n|+(1-\frac{\rho}{2})|w_{n-1}|$.

Finally, joining all the inequalities we conclude that
$$
|w_n|\leq \max \bigg\{ |\lambda _n|+\frac{1}{\rho} |\eta _n|\bigg\} +\left( 1-\frac{\rho}{2}\right) |w_{n-1}|
$$  
which implies that $w\in c_0$.
\end{proof}

As similar results hold for the left derivative, we deduce the following theorem.

\begin{thm}\label{mainD}
Let $w\notin c_0$.
Assume that the decomposition satisfies that 
$D_{n+1}\cap I\neq \emptyset$ for every $I\in \mathcal{F}_n$, and
that  either $\rho >\frac{1}{2}$ or $\alpha _{n+1}\leq \rho \alpha _n$ 
for every $n$.
If $x\in D$, then $T_w$ has not finite lateral derivatives
at $x$.
\end{thm}

The following example shows the necessity of a condition
as $D_{n+1}\cap I\neq \emptyset$ for every $I\in \mathcal{F}_n$.

\begin{example}
Let $w\notin c_0$ be defined as $w_0=0$, $w_{3k-2}=1$, $w_{3k-1}=-1$ for every $k$,
$w_3=-2^{-1}$, and $w_{3k}=2^{-k}$ for every $k>1$. Let us consider a decomposition satisfying
$D_{3k-2}=D_{3k-1}$, and $D_{3(k+1)-2}=D_{3k}=D_{3k-1}\cup \widetilde{D}_{3k-1}$ for every $k$.
It satisfies $\widetilde{D}_n\subset D_{n+1}\cup \widetilde{D}_{n+1}$ for every $n$. We have that $T'_w(x)=0$
for every $x\in D_1$.
\end{example}
\begin{proof}
It is enough to observe that 
$$
T_w(z)=-\frac{1}{2}g_3(z)+\sum_{k=2}^{\infty}\frac{1}{2^k}g_{3k}(z)
$$
for every $z$. As a consequence of the results obtained in \cite{FGGLL1}, where a deep study of the case $w\in \ell^1$ is developed, we conclude
$$
T'^+_w(x)=-\frac{1}{2}+\sum_{k=2}^{\infty}\frac{1}{2^k}=0
$$
and
$$
T'^-_w(x)=\frac{1}{2}+\sum_{k=2}^{\infty}\left( -\frac{1}{2^k}\right)=0.
$$
\end{proof}

If $\mathcal{L}$ denotes the Lebesgue measure on $\mathbb{R}$, the Example \ref{ejemploD} below will illustrate how the existence of $\rho \in (0,1]$ such that 
$\mathcal{L}(I)\geq \rho \alpha _n$ for every $I\in \mathcal{F}_n$
should be required in order to get the nowhere derivability of the function $T_w$.

We may also study the subdifferential and the superdifferential. Firstly, we consider the situation in which the weights are nonnegative,
without restrictions on the decomposition. We have the following result that 
extends the well known result for the original Takagi function (see \cite{G} and \cite{FGG}).

\begin{prop}\label{positivosxenD}
Assume that $w\notin c_0$ and $w_k\geq 0$ for every $k$. 
If $x\in D$, then the function $T_w(z)-\zeta z$ has a local minimum at $x$ for every $\zeta\in \mathbb{R}$. In particular, 
$$
\partial T_w(x)=\mathbb{R}.
$$
\end{prop}
\begin{proof}
If $x\in D$, there exists $n_0$ such that $x\in D_{n_0}$ and $x\notin D_k$ provided that $k<n_0$. 
For every $\zeta\in\mathbb{R}$ let $n$ be such that 
$$
\sum_{k=n_0}^{n}w_k>\sum_{k=1}^{n_0-1}w_k+|\zeta|.
$$
On the other hand, if $|h|<\frac{1}{3}\rho \alpha_n$ then $g_k(x+h)=|h|$ for every $k$, $n_0\leq k \leq n$. Hence, 
\begin{align*}
&T_w(x+h)-T_w(x)-\zeta h \geq \sum_{k=1}^{n_0-1}w_k\left ( g_k(x+h)-g_k(x)\right )+\sum_{k=n_0}^{n}w_k g_k(x+h)-|\zeta||h|\\
&\geq |h|\left [ \sum_{k=n_0}^{n}w_k - \left (\sum_{k=1}^{n_0-1}w_k  +|\zeta|\right)\right ]\geq 0.
\end{align*}
\end{proof}

Observe that in the proof  we only require that $w\notin \ell ^1$.

\begin{cor}
Assume that $w\notin \ell ^1$ and $w_k\geq 0$ for every $k$. If $x\in D$, then $T_w$ is not derivable at $x$.
\end{cor}

The next results involve arbitrary sequences $w$, and we require that the decomposition
satisfies $\alpha _{n+1}\leq \frac{\rho \alpha _n}{2}$ for every $n$. This is a mild restriction, and the functions
of the Takagi-Van der Waerden Class satisfy it.

\begin{prop}\label{subdiferencialnovacia}
Assume that $w_0=0$ and $\alpha _{n+1}\leq \frac{\rho \alpha _n}{2}$ for every $n$.
If $x\in D_1$ and  $\partial T_w(x)\neq \emptyset$, then
$$
a:=\liminf_n \sum_{k=1}^nw_k\geq 0.
$$
\end{prop}
\begin{proof}
For every $n$, we denote $y_n=\min \{ y\in D_n:x<y\}$. Observe that if $d_k\in \widetilde{D}_k$ is the midpoint
between $x$ and $y_k$, then $d_k-x\geq \frac{\rho \alpha _k}{2}\geq \alpha _{k+1}$.
This implies that $g_k(y_n)=y_n-x$ provided that $k<n$. Hence
\begin{align*}
&d_+T_w(x)=\liminf_{t\to 0^+}\frac{T_w(x+t)-T_w(x)}{t}=\liminf_{t\to 0^+}\frac{T_w(x+t)}{t}\\&\leq
\liminf_n\frac{T_w(y_n)}{y_n-x}=\liminf_n\sum_{k=1}^{n-1}w_k\frac{g_k(y_n)}{y_n-x}=
\liminf_n\sum_{k=1}^{n-1}w_k=a.
\end{align*}
Now we denote $x_n=\max \{ y\in D_n:y<x\}$. Again, we have that if $d_k\in \widetilde{D}_k$ is the midpoint
between $x$ and $x_k$, then $x-d_k\geq \alpha _{k+1}$. Consequently
\begin{align*}
&D^-T_w(x)=\limsup_{t\to 0^-}\frac{T_w(x+t)-T_w(x)}{t}
=\limsup_{t\to 0^-}\frac{T_w(x+t)}{t}\\ &\geq
\limsup_n\frac{T_w(x_n)}{x_n-x}=-\liminf_n\sum_{k=1}^{n-1}w_k\frac{g_k(x_n)}{x-x_n}=
-\liminf_n\sum_{k=1}^{n-1}w_k=-a.
\end{align*}
Therefore, if $\partial T_w(x)\neq \emptyset$ then $-a\leq a$ which implies $a\geq 0$.
\end{proof}

\begin{prop}
Assume that $w_0=0$ and $\alpha _{n+1}\leq \frac{\rho \alpha _n}{2}$ for every $n$.
If $x\in D_1$ and  $\partial ^+T_w(x)\neq \emptyset$, then
$$
b:=\limsup_n \sum_{k=1}^nw_k\leq 0.
$$
\end{prop}
\begin{proof}
Assume that $\partial ^+T_w(x)\neq \emptyset$, this implies that 
$\partial (-T_w)(x)\neq \emptyset$, hence
$$
-b=\liminf_n \sum_{k=1}^n(-w_k)\geq 0
$$
and consequently $b\leq 0$
since 
$$
(-T_w)(x)=\sum_{n=0}^{\infty}(-w_n)g_n(x).
$$
\end{proof}

If $x\in D$, it is possible that $x\in \widetilde{D}$ too, but if this is the situation,
then $x\in \widetilde{D}_k$ implies $k<n_0$ where $n_0$ is the smallest index $n$ 
satisfying $x\in D_n$. We may write
$$
T_w=\sum_{k<n_0, x\notin \widetilde{D}_k}w_kg_k+\sum_{k<n_0, x\in \widetilde{D}_k}w_kg_k
+\sum_{k={n_0}}^{\infty}w_kg_k.
$$ 
The function $\sum_{k<n_0, x\notin \widetilde{D}_k}w_kg_k$ is derivable at $x$, and  therefore it does not modify
the derivability character of $T_w$. Hence, we may assume that $x\in \widetilde{D}_k$ for
every $k<n_0$. We have the following results:

\begin{prop}\label{subxenDyDtilde}
Assume that $\alpha _{n+1}\leq \frac{\rho \alpha _n}{2}$ for every $n$.
If $x\in D_{n_0}$, $x\in \widetilde{D}_k$ for every $k<n_0$, and  $\partial T_w(x)\neq \emptyset$, then
$$
\liminf_n \sum_{k=n_0}^nw_k-\sum_{k=0}^{n_0-1}w_k\geq 0.
$$
\end{prop}
\begin{proof}
It is routine to check that both $d_+T_w(x)$ and $-D^-T_w(x)$ are less than or equal to
$$
\liminf_n \sum_{k=n_0}^nw_k-\sum_{k=0}^{n_0-1}w_k,
$$
hence it is necessarily  greater than or equal to $0$ provided that $\partial T_w(x)\neq \emptyset$.
\end{proof}

Similarly,

\begin{prop}\label{supxenDyDtilde}
Assume that $\alpha _{n+1}\leq \frac{\rho \alpha _n}{2}$ for every $n$.
If $x\in D_{n_0}$, $x\in \widetilde{D}_k$ for every $k<n_0$, and  $\partial ^+T_w(x)\neq \emptyset$, then
$$
\limsup_n \sum_{k=n_0}^nw_k-\sum_{k=0}^{n_0-1}w_k\leq 0.
$$
\end{prop}

These results allow us  to immediately deduce the non derivability of $T_w$. Anyway, observe  
that this result is weaker than Theorem \ref{mainD}.

\begin{cor}\label{xenD}
Assume that $\alpha _{n+1}\leq \frac{\rho \alpha _n}{2}$ for every $n$.
If $w\notin c_0$ and $x\in D$, then $T_w$ is not derivable
at $x$.
\end{cor}
\begin{proof}
If $T'_w(x)$ exists then $\partial T_w(x)=\partial ^+T_w(x)\neq \emptyset$. Hence
$\sum_{n\geq n_0}w_n$ converges, which is impossible since $w\notin c_0$.
\end{proof}

\section{Nowhere derivability}

We begin with an example that shows the  necessity of a condition such as the existence
of $\rho$, in order to obtain positive results. We define the function on $[-1,1]$ instead of $[0,1]$
for the sake of simplicity.

\begin{example}\label{ejemplonoD}
Let us consider the sets $D_0^+=\{1\}$ and $D_1^+=\left \{\frac{2}{3},1\right \}$. For every integer $n\geq 1$ we define the sets 
$$
D_{2n}^+=\left \{\frac{k}{2^n}\in (0,1]: k\in\mathbb{Z}\right \}\cup D_{2n-1}^+
$$
and
$$
 D_{2n+1}^+=\left \{\frac{k}{2^n}-\frac{1}{3^{n+1}}\in (0,1]: k\in\mathbb{Z}\right \} \cup D_{2n}^+.
$$
For all $n\geq 0$ we also define $D_n^-=\{-x:x\in D_n^+\}$ and we consider the set $D_n=D_n^+\cup D_n^-$. Let $w\notin c_0$ be defined as $w_{2n}=1$ and $w_{2n+1}=-1$ for every $n$. Then, $T_w$ is derivable at $0$ and $T'_w(0)=0$.
\end{example}
\begin{proof}
We may rewrite the function $T_w:[-1,1]\to\mathbb{R}$ as
$$
T_w(x)=\sum_{n=0}^{\infty}H_n(x)
$$
where $H_n(x) = w_{2n}g_{2n}(x) + w_{2n+1}g_{2n+1}(x)$. It is immediate that, for every $n$, $\left |H_{n}(x)\right |\leq \frac1{3^{n+1}}$ for all $x\in[-1,1]$ and $H_n(x)=\frac{1}{3^{n+1}}$ provided that $x\in [-\frac{1}{2^n}+\frac{1}{3^{n+1}},\frac{1}{2^n}-\frac{1}{3^{n+1}}]$. 
If $0<|h|<\frac12$ there exists an integer $n$ such that $\frac{1}{2^{n+1}}\leq |h|< \frac{1}{2^n}$ and as $H_k(0)=H_k(h)$ for all $ 0 \leq k\leq n-1$ we have
$$
\left|\frac{T_w(h)-T_w(0)}{h}\right|=\frac{1}{|h|}\left| \sum_{k=n}^{\infty}H_k(h)-\sum_{k=n}^{\infty} \frac{1}{3^{k+1}}\right|\leq 2^{n+2}\sum_{k=n}^{\infty} \frac{1}{3^{k+1}}=2\left (\frac{2}{3}\right )^n.
$$
Letting $h$ to zero and therefore, $n$ to infinity, we obtain that $T'(0)=0$.
\end{proof}

\begin{example}\label{ejemploD}
If we modify Example \ref{ejemplonoD}, defining $D_0^+=\{0,1\}$,
we also have that $T_w$ is  derivable at $0\in D$.
\end{example}

Now we assume that $x\notin D$. In this section  we will denote $a_n=\max \{ y\in D_n: y<x\}$, 
$b_n=\min \{ y\in D_n: x<y\}$ and $c_n$ the midpoint of $(a_n,b_n)$.

The problem in the general situation is 
that we cannot control the position of the midpoints $c_k$ with respect to $(a_n,b_n)$ for $k<n$.
This forces us to require extra restrictions on the decomposition, however they are quite mild and 
allow us to generalize 
Kono's result for the Takagi Class (see \cite{Kono}), since the natural
decomposition of the dyadic numbers satisfy $\widetilde{D}_n\subset D_{n+1}$.  

\begin{thm}\label{rpar}
Assume that $w\notin c_0$. If the decomposition satisfies  for every $n$ that either
\begin{enumerate}
  \item  $\widetilde{D}_n\subset D_{n+1}$, or 
  \item  $I\cap D_{n+1}\neq \emptyset$ for every $I\in \mathcal{F}_n$ and $\widetilde{D}_n\subset \widetilde{D}_{n+1}$,
\end{enumerate}
then $T_w$ is nowhere derivable.
\end{thm}
\begin{proof}
It is enough to prove that $T'_w(x)$ does not exist for every $x\notin D$ since  Proposition \ref{sencillo}
gives us the result when $x\in D$.
If $x\in \widetilde{D}$ we may assume without loss of generality that $x\in \widetilde{D}_1$ and consequently that 
$x\in \widetilde{D}_n$ for every $n$, since otherwise it would belong to $D$. If this is the case, then
$x$ is the midpoint of $(a_n,b_n)$ for every $n$, which implies that $a_{n}<a_{n+1}<x<b_{n+1}<b_{n}$ for every
$n$ since  $I\cap D_{n+1}\neq \emptyset$ for every $I\in \mathcal{F}_n$. Hence, if $T'_w(x)$ exists then
$$
\frac{T_w(b_{n+1})-T_w(b_n)}{b_{n+1}-b_n}=\sum_{k=1}^nw_k\frac{g_k(b_{n+1})-g_k(b_n)}{b_{n+1}-b_n}
=-\sum_{k=1}^nw_k,
$$
and consequently, we have that
$$
-\lim_n\sum_{k=1}^{n}w_k=T'_w(x)
$$
by Lemma \ref{lema1} since 
$$
\frac{b_{n+1}-x}{b_n-b_{n+1}}\leq \frac{\alpha _{n+1}}{\rho \alpha _{n+1}}=\frac{1}{\rho} .
$$
This implies that $\sum w_k$ converges, which is impossible.

Therefore, we may assume that $x\notin \widetilde{D}$. For every $n$, we have
that if $c_k\notin (a_n,b_n)$ for every $k<n$ then
$$
\frac{T_w(b_n)-T_w(a_n)}{b_n-a_n}=\sum_{k=1}^{n-1}w_k\frac{g_k(b_n)-g_k(a_n)}{b_n-a_n}
=\sum_{k=1}^{n-1}w_kg'_k(x).
$$

Alternatively, if $c_k\in (a_n,b_n)$ for some $k<n$ then $c_k=c_{k+1}=\dots =c_n$, since we are under hypotheses $(2)$ necessarily,
this implies in particular that 
$$
a_k<a_{k+1}<\dots <a_n<b_n<\dots <b_{k+1}<b_k.
$$
If $c_n<x$ then
$$
\frac{T_w(b_n)-T_w(b_{n-1})}{b_n-b_{n-1}}=\sum_{k=1}^{n-1}w_k\frac{g_k(b_n)-g_k(b_{n-1})}{b_n-b_{n-1}}
=\sum_{k=1}^{n-1}w_kg'_k(x)
$$
and 
$$
\frac{b_n-x}{b_{n-1}-b_n}\leq \frac{\alpha _n}{\rho \alpha _n}=\frac{1}{\rho},
$$
meanwhile if $x<c_n$ then
$$
\frac{T_w(a_n)-T_w(a_{n-1})}{a_n-a_{n-1}}=\sum_{k=1}^{n-1}w_k\frac{g_k(a_n)-g_k(a_{n-1})}{a_n-a_{n-1}}
=\sum_{k=1}^{n-1}w_kg'_k(x)
$$
and
$$
\frac{x-a_{n-1}}{a_{n}-a_{n-1}}=\frac{x-a_n}{a_{n}-a_{n-1}}+1
\leq \frac{\alpha _n}{\rho \alpha _n}+1=\frac{1}{\rho}+1.
$$
If $T'_w(x)$ exists, then taking $u_n=b_n$ or $a_n$, and $v_n=a_n,b_{n-1}$ or $a_{n-1}$ accordingly,
invoking Lemma \ref{lema1},
we have that 
$$
T'_w(x)=\lim_n \frac{T_w(u_n)-T_w(v_n)}{u_n-v_n}=\sum_{k=1}^{\infty}w_kg'_k(x)
$$
which is a contradiction.
\end{proof}

This theorem covers 
the Takagi-Van der Waerden Class that  we introduced in Example \ref{ejemplo1},
but it also includes a more general situation appearing in the following example. 

\begin{example}\label{ejemplo2}
Let us consider $\pmb{r}=(r_n)_n\subset \mathbb{N}$ an  strictly increasing sequence 
where $r_{n}$ divides $r_{n+1}$ for every $n$ and $r_1=1$. 
We define the function $f_{\pmb{r}}:[0,1]\to\mathbb{R}$ as follows
$$
f_{\pmb{r}}(x)=\sum_{n=1}^{\infty}\frac{1}{r_n}\phi(r_nx)
$$
If we define the functions $f_{\pmb{r},w}$ in the same way as in the Example \ref{ejemplo1}, we have that these functions are 
nowhere differentiable whenever $w\not\in c_0$.
\end{example}
\begin{proof}
It is immediate to see that $D_n=\{kr^{-1}_n\in [0,1]:k\in\mathbb{Z}\}, \alpha_n = \frac{1}{r_n}$ and $\rho = 1$. We may write $r_{n+1}=\beta_n r_{n}$. If $\beta_n$ is even then we have the situation $\widetilde{D}_n\subset D_{n+1}$, meanwhile if $\beta_n$ is odd we obtain that $\widetilde{D}_n\subset \widetilde{D}_{n+1}$. Thus we have a decomposition which satisfies the conditions described in Theorem \ref{rpar}.
\end{proof}

\begin{rmk}
If $\rho =1$ then there exists a $\pmb{r}$ as above, such that 
$T_w=f_{\pmb{r},w}$.
\end{rmk}

When the weights are nonnegative we do not require any restriction on the decomposition, let us see it.

\begin{thm}\label{positivosnoenD}
If $w\notin c_0$, $w_k\geq 0$ for every $k$, and $x\notin D$, then $\partial T_w(x)=\emptyset$.
\end{thm}
\begin{proof}
We denote $a_n<x<b_n$ satisfying $a_n,b_n\in D_n$, $(a_n,b_n)\in \mathcal{F}_n$
as usual. We have
\begin{align*}
&d_+T_w(x)=\liminf_{t\to 0^+}\frac{T_w(x+t)-T_w(x)}{t}\leq 
\liminf_n\frac{T_w(b_n)-T_w(x)}{b_n-x}\\
&\leq \liminf_n\sum_{k=1}^nw_k\frac{g_k(b_n)-g_k(x)}{b_n-x}\leq 
\liminf_n\sum_{k=1}^nw_kg'_k(x),
\end{align*}
since 
$$
\frac{g_k(b_n)-g_k(x)}{b_n-x}=g'_k(x)
$$
provided that $g'_k(x)=-1$. (Observe that this happens whenever $c_k<x$).
Similarly,
\begin{align*}
&D^-T_w(x)=\limsup_{t\to 0^-}\frac{T_w(x+t)-T_w(x)}{t}\geq 
\limsup_n\frac{T_w(a_n)-T_w(x)}{a_n-x}\\
&\geq \limsup_n\sum_{k=1}^nw_k\frac{g_k(a_n)-g_k(x)}{a_n-x}\geq 
\limsup_n\sum_{k=1}^nw_kg'_k(x).
\end{align*}
We conclude that $\partial T_w(x)\neq \emptyset$ implies
$$
\limsup_n\sum_{k=1}^nw_kg'_k(x)\leq D^-T_w(x)\leq d_+T_w(x)
\leq \liminf_n\sum_{k=1}^nw_kg'_k(x)
$$
and the convergence of the series $\sum w_kg'_k(x)$, which is
impossible since $w\notin c_0$. 
\end{proof}

\begin{cor}
If $w\notin c_0$ and $w_k\geq 0$ for every $k$, then $T_w$ is nowhere derivable.
\end{cor}

The following example shows that Proposition \ref{positivosxenD} and Corollary \ref{positivosnoenD} are not longer true when considering arbitrary weights.
\begin{example}
Let us consider the function $f:[0,1]\to \mathbb{R}$ defined as $f(x)=-T(x)$ where $T$ denotes the Takagi function. From 
Proposition \ref{positivosxenD}, we have that if $x\in D$ then $\partial T(x)=\mathbb{R}$, and consequently
$$
\partial^+ f(x)=-\partial (-f)(x)=-\partial T(x)=\mathbb{R}.
$$
Hence, $\partial f(x)=\emptyset$. On the other hand, if $x\notin D$ and we consider its binary expansion
$$
x=\sum_{n=1}^{\infty}\frac{\varepsilon_n}{2^n}, \quad \varepsilon_n\in\{0,1\},
$$
we have the following result for the superdifferential of the Takagi function (see \cite{FGG2}, Theorem 2.6):
\begin{prop}
If $x\notin D$ and there exists $m\in\mathbb{Z}$, $m\geq 1$, such that $\varepsilon_n+\varepsilon_{n+1}=1$ for all $n>m$, then 
$$
\partial^+ T(x)=\left \{
\begin{array}{lc}
m-2\sum_{k=1}^m \varepsilon_k +[-1,0]&\text{if }\varepsilon_{m+1}=1,\\
&\\
m-2\sum_{k=1}^m \varepsilon_k +[0,1]&\text{if }\varepsilon_{m+1}=0.\\
\end{array}
\right .
$$
\end{prop}
In view of this result, it is enough to consider the point $x=11010\overline{10}$, and consequently, $\partial f(x)=[1,2]$.
\end{example}

\end{document}